\documentclass{amsart}

    \usepackage[utf8]{inputenc}
    \usepackage{verbatim}
    \usepackage{amsfonts}
    \usepackage{amsmath}
    \usepackage{amsthm}
     \usepackage{amsaddr}
    \usepackage{amssymb} 
    \usepackage{mathtools} 
        \newcommand{\xtwoheadrightarrow}[2][]{\xrightarrow[#1]{#2}\mathrel{\mkern-14mu}\rightarrow} 
    \usepackage{stmaryrd} 
    \usepackage{enumitem} 
    
    \usepackage[maxbibnames=99]{biblatex}
        \DeclareFieldFormat{title}{\mkbibquote{#1}}

    \usepackage{hyperref}
    \usepackage[noabbrev,nameinlink,capitalise]{cleveref}
        \crefname{subsection}{Subsection}{Subsections}
        \crefname{subsection}{Subsection}{Subsections}
    
    \usepackage{subcaption}
    \usepackage{color}
    \usepackage[dvipsnames]{xcolor}
    \usepackage{graphicx}\graphicspath{{./images/}}
    \usepackage{tikz-cd} 
    \usepackage{tikz}
        \usetikzlibrary{
            calc,
            decorations.pathmorphing,
            }
        \tikzset{every picture/.style=thick}
        \tikzset{curvy/.style={decorate, decoration=snake,segment length=.5cm}}

    \usepackage{orcidlink}
    
    \theoremstyle{plain}
        \newtheorem{theorem}{Theorem}[section]
        \newtheorem{corollary}[theorem]{Corollary}
        \newtheorem{lemma}[theorem]{Lemma}
        \newtheorem{proposition}[theorem]{Proposition}
    \theoremstyle{definition}
        \newtheorem{definition}[theorem]{Definition}

    \theoremstyle{remark}
        \newtheorem{remark}[theorem]{Remark}
        \newtheorem{question}[theorem]{Question}

    \title{Orderable Thompson-like groups arising from Ore categories}
    \author{Davide Perego\,$^{\ast}$ and Matteo Tarocchi\,$^{\dagger}$}
    \date{}
    
    \thanks{
    Both authors acknowledge support from the research grant PID2022-138719NA-I00 (Proyectos de Generación de Conocimiento 2022) financed by the Spanish Ministry of Science and Innovation.
    The first author is supported by the Swiss Government Excellence Scholarship and from Swiss NSF grant 200020-200400.
    The second author is supported by the French project GoFR (ANR-22-CE40-0004).
    Both the authors are members of the Gruppo Nazionale per le Strutture Algebriche, Geometriche e le loro Applicazioni (GNSAGA) of the Istituto Nazionale di Alta Matematica (INdAM)%
    }



    \address{
        ${\ast}$
        Section de mathématiques, Université de Genève, rue du Conseil-Général 7-9, 1205 Genève, Switzerland\\
        \,\href{https://orcid.org/0009-0008-5939-7942}{0009-0008-5939-7942}\,\orcidlink{0009-0008-5939-7942}\\
        \,\href{mailto:dperego9@gmail.com}{dperego9@gmail.com}
        }
    
    \address{
        ${\dagger}$
        Université Paris-Saclay, CNRS, Laboratoire de mathématiques d’Orsay, Orsay, France, EU \& Université de Rennes, CNRS, IRMAR, Rennes, France, EU\\
        \,\href{https://orcid.org/0000-0003-3668-4654}{0000-0003-3668-4654}\,\orcidlink{0000-0003-3668-4654}\\
        \,\href{mailto:matteo.tarocchi.math@gmail.com}{matteo.tarocchi.math@gmail.com}
        }

    \newcommand{\F}{\mathcal{F}}
    \newcommand{\G}{\mathcal{G}}
    \newcommand{\D}{\mathcal{D}}
    \renewcommand{\P}{\mathcal{P}}
    \newcommand{\R}{\mathbf{R}}
    \newcommand{\C}{\mathcal{C}}
    \newcommand{\B}{\mathcal{B}}
    \newcommand{\PB}{\mathcal{PB}}
    \newcommand{\FH}{\mathcal{FH}}
    \newcommand{\BH}{\mathcal{BH}}
    \newcommand{\Ob}{\mathrm{Ob}}

\begin{document}

\begin{abstract}
We give sufficient conditions for left- and bi-orderability of fundamental groups of Ore categories in terms of indirect factors, including Thompson groups and many of their generalizations.
Besides recovering known results, we prove that braided groups of fractions of digit rewriting systems (which generalize braided Thompson groups to the wider setting of topological full groups of edge shift) are left-orderable, and that their purely braided counterparts are bi-orderable.
In particular, the braided Houghton groups are left-orderable.
\end{abstract}

\maketitle


\section{Introduction}

Many generalizations of Thompson groups have the following form:
there is a notion of \textit{expansion} (classically consisting of attaching a caret to a complete rooted binary tree) and a notion of \textit{isomorphism} of objects arising from such expansions (for Thompson's groups $F$, $T$ and $V$, these are trivial, cyclic, or any permutation of leaves, respectively).
The Thompson-like group is then a group of ``fractions'', each given by two expansions and an isomorphism between the objects determined by such expansions (the typical tree pair diagrams for $F$, $T$ and $V$ model precisely this).
As described in \cite{Witzel}, this is formalized by the notion of fundamental groups of certain categories $\F\bowtie\G$:
the factor $\F$ determines the expansions and $\G$ determines the isomorphisms;
the category is assumed to satisfy some properties to ensure that the elements of the fundamental group have their ``fraction-like'' form:
expansion--isomorphism--reduction (\cref{rmk:normal:forms}).
The required properties make $\F\bowtie\G$ a so-called Ore category.

\cref{thm:main} shows how, for certain products $\F\bowtie\G$, if the fundamental group of the first factor $\F$ is itself left- or bi-orderable and if $\G$ has a morphism $\Phi$ to a discrete groupoid comprised of left- or bi-ordered groups and satisfies certain conditions of compatibility with the ordering of $\F$ (\cref{def:compatible:ordering}), then the fundamental group of $\F\bowtie\G$ is left- or bi-orderable, respectively.

As an application, we let $\F$ be a category modeling a rewriting system whose rewriting rules are replacements of subwords with digits and $\G$ be a groupoid of (pure) braids on the digits of finite words.
The fundamental groups of $\F$ are instances of diagram groups and those of $\F\bowtie\G$ are (purely) braided versions of such groups, which can be thought of as braided versions of topological full groups of edge shifts \cite{Matui}.
They include the braided Thompson groups \cite{DehornoyThompsonOrder,,BrinBraided} and the braided Houghton groups introduced in \cite{braidedHoughtonThesis} and studied in \cite{BraidedHoughton} as mapping class groups.
In particular, we show that the braided Houghton groups are left-orderable, which, as far as the authors are aware, is a new result.

In most previously considered cases of fundamental groups of categories $\F\bowtie\G$, the factor $\G$ was a discrete groupoid (i.e., just a collection of disjoint groups).
Our examples show additional complexity, since $\G$ will generally not be discrete.

The remainder of the Introduction is devoted to stating \cref{thm:main} and briefly explaining the necessary hypotheses.
We refer to the upcoming \cref{sub:Ore:cat} for the definitions of right-Ore categories, indirect products $\F\bowtie\G$ and the actions of $\F$ and $\G$ on each other (denoted by $g^f$ and $g \cdot f$).

Recall that a groupoid $\G$ is \textbf{discrete} if $\G(x,y) = \emptyset$ whenever $x \neq y$.

\begin{definition}
\label{def:discrete:ordering}
Let $\D$ be a discrete groupoid.
We say that a \textbf{discrete left-ordering} (or \textbf{bi-ordering}) $<_\D$ of $\D$ is a collection of left-orderings (or bi-orderings) $<_x$ of the groups $\D(x,x)$ for all objects $x \in \Ob(\D)$.
\end{definition}

This is not related to the notion known in the literature as \textit{groupoid ordering}.

We will denote by $\mathrm{C}\G(x)$ the connected component of $\G$ that contains the object $x \in \Ob(\G)$, which is a full subgroupoid of $\G$.

\begin{definition}
\label{def:rho:discrete:ordering}
Let $\G$ be a groupoid, let $\D$ be a discrete groupoid and let $\prec_k$ be a left-ordering on $\D(k,k)$, for each $k\in\Ob(\D)$.
Assume that $\Phi$ is a groupoid functor $\G \to \D$ such that $\G(x,y)\neq\emptyset \implies \Phi(x)=\Phi(y)$.
A \textbf{$\Phi$-discrete left-ordering} of $\G$ is a collection of sets
$$ P_{\mathrm{C}\G(x)} \coloneq \{ g \in \mathrm{C}\G(x) \mid 1_{\D(\Phi(x),\Phi(x))} \prec_{\Phi(x)} \Phi(g) \}. $$
If $\prec_k$ are bi-orderings, we will say that this collection of sets is a \textbf{$\Phi$-discrete bi-ordering}.
Moreover, we will say that a non-identity element is \textbf{positive} if it belongs to a such a set and that it is \textbf{negative} otherwise.
\end{definition}

\begin{definition}
\label{def:compatible:ordering}
Let $\F\bowtie\G$ be an indirect product.
A $\Phi$-discrete ordering on $\G$ is \textbf{compatible} with $\F$ if
$$ g \in P_{\mathrm{C}\G(x)} \implies g^f \in P_{\mathrm{C}\G(y)},\ \forall f \in \F(x,y),\ \forall g \in \G(-,x). $$
\end{definition}

For bi-orderability, we will focus on discrete groupoids satisfying the following condition, which generalizes the definition of a pure cloning system \cite{CloningSystemsOrderable}.

\begin{definition}
\label{def:pure}
An indirect product $\F\bowtie\G$ is \textbf{pure} if $\G$ is a discrete groupoid and, for all $g\in\G(x,x)$ and $f\in\F(x,-)$, one has that $g \cdot f = f$.
\end{definition}

\begin{theorem}
\label{thm:main}
Given a right-Ore indirect product $\F\bowtie\G$, where $\G$ is a groupoid, we have the following.
\begin{enumerate}
    \item If $\Pi(\F,\varepsilon)$ is left-orderable and $\G$ has a $\Phi$-discrete left-ordering that is compatible with $\F$, then the group $\Pi(\F\bowtie\G,\varepsilon)$ is left-orderable.
    \item If $\F\bowtie\G$ is pure, $\Pi(\F,\varepsilon)$ is bi-orderable and $\G$ has a $\Phi$-discrete bi-ordering that is compatible with $\F$, then $\Pi(\F\bowtie\G,\varepsilon)$ is bi-orderable.
\end{enumerate}
\end{theorem}

This recovers Ishida's result \cite[Proposition 1.1]{CloningSystemsOrderable} and has the additional following implication about groups of braided fractions of digit rewriting systems.

\begin{corollary}
\label{cor:applications}
The group of fractions of a digit rewriting system is left-orderable and the group of purely braided fractions of a digit rewriting system is bi-orderable.
In particular, the braided Houghton groups are left-orderable.
\end{corollary}

This will be shown throughout \cref{sub:left:orderable:examples,,sub:left:orderable:braided:houghton}.

\begin{question}
Can the results of \cref{cor:applications} be extended to the case of more general string rewriting systems (which model braided versions of diagram groups)?
\end{question}

Note that, for $\F$ a generic string rewriting system, the Ore category generated by $\F$ and $\G$ is generally not an indirect product.


\section{Background and main definitions}
\label{sec:background}

This section very briefly introduces the relevant background.

\subsection{Orderable groups}

A group is \textbf{left-orderable} if there exists a linear order on $G$ that is invariant under left multiplication, and it is \textbf{bi-orderable} if the linear order is invariant under both left and right multiplication.

Equivalently, a group $G$ is left-orderable if and only if it has a \textbf{positive cone}, which is a subset $P \subseteq G$ such that:
(1)~$P \cdot P \subseteq P$,
(2)~$P \cap P^{-1} = \emptyset$,
(3)~$P \cup P^{-1} = G \setminus \{1\}$.
If the positive cone $P$ is invariant under conjugacy, then it $G$ is bi-orderable.

See for example \cite{OrderableGroups} for more information on left- and bi-orderable groups.

\subsection{Fundamental groups of right-Ore categories}
\label{sub:Ore:cat}

Let us briefly recall the basic definitions about fundamental groups of right-Ore categories and indirect products.
We refer to \cite{Witzel} for all the details.
As done in \cite{Witzel}, we will write $\C(x,y)$ to denote the set of morphisms of $\C$ from $y$ to $x$.

A category $\C$ has \textit{common right multiples} if, $\forall a_1, a_2 \in \C(x,-)$, $\exists b_1, b_2 \in \C$ such that $a_1b_1=a_2b_2$.
It is \textit{cancellative} if $ab=ac \implies b=c$ and $ab=cb \implies a=c$.
The category $\C$ is \textbf{right-Ore} if it has common right multiples and is cancellative.
For more detail, see \cite[Subsection 1.3]{Witzel}.

If $\F$ and $\G$ are subcategories of $\C$, then $\C$ is an \textbf{indirect product} $\F\bowtie\G$ (also known as Zappa-Sz\'{e}p product in the case of groups) if, $\forall f\in\F(x,-), g\in\G(-,x)$, $\exists! f'\in\F, g'\in\G$ for which $gf=f'g'$.
Writing $f'=g \cdot f$ and $g'=g^f$ determines actions of $\G$ and $\F$ on each other.
Indirect products are also defined externally by actions satisfying certain properties, and by \cite[Lemma 4.1]{Witzel} an external indirect product is naturally isomorphic to an internal product.
For more detail, see \cite[Section 4]{Witzel}.
The following is point (i) of \cite[Proposition 4.5]{Witzel}.
Although the statement of \cite[Proposition 4.5]{Witzel} assumes that the groupoid $\G$ is discrete, this is not required for the proof of its first point.

\begin{lemma}
\label{lem:product}
An indirect product $\F\bowtie\G$ is right-Ore if $\G$ is a groupoid and $\F$ is right-Ore, has no non-trivial invertibles and has injective action on $\G$.
\end{lemma}

We will write $\Pi(\C,\varepsilon)$ for the fundamental group of $\C$ based at an object $\varepsilon$.
When the choice of basepoint is irrelevant, we will often omit it and simply write $\Pi(\C)$.

\begin{remark}
\label{rmk:normal:forms}
Since $\F\bowtie\G$ is a right-Ore category, it has common right-multiples, so every element of the fundamental group $\Pi(\F\bowtie\G)$ can be written as $ab^{-1}$ for some $a,b\in\F\bowtie\G(x,-)$ and some $x \in \Ob(\F\bowtie\G)$ (see Theorem 1.6 of \cite{Witzel} and the discussion that follows it).
Since $a=f_1g_1$ and $b=f_2g_2$ for $f_1,f_2\in\F$ and $g_1,g_2\in\G$, we then have $ab^{-1}=f_1gf_2^{-1}$, where $g=g_1g_2^{-1}$.
Thus, every element of $\Pi(\F\bowtie\G,\varepsilon)$ can be written as $f_1gf_2^{-1}$ for some $f_1,f_2\in\F$ and $g\in\G$.
We will write $TgS^{-1}$, where $S\in\F(\varepsilon,y)$, $T\in\F(\varepsilon,x)$ and $g\in\G(x,y)$.
\end{remark}

\begin{remark}
\label{rmk:composition}
The composition of two elements $TgS^{-1}$ and $T'hS'^{-1}$ is
$$TgS^{-1} T'hS'^{-1} = TgB B^{-1}S^{-1} T'A A^{-1} hS'^{-1},$$
where $A,B \in \F$ are such that $T'A = SB$ and exist by $\F$ having common right-multiples.
The previous equation continues as follows:
\begin{gather*}
TgB B^{-1}S^{-1} T'A A^{-1} hS'^{-1} = TgB(SB)^{-1}(T'A) A^{-1}hS'^{-1} = TgBA^{-1}hS'^{-1} =\\
=TgB(h^{-1}A)^{-1}S'^{-1} = T(g\cdot B) g^B [(h^{-1}\cdot A) (h^{-1})^A]^{-1} S'^{-1} =\\
= T (g\cdot B) g^B [(h^{-1})^A]^{-1} (h^{-1}\cdot A)^{-1} S'^{-1}.
\end{gather*}
\end{remark}


\section{Proof of \texorpdfstring{\cref{thm:main}}{Theorem 1.5}}

Throughout this section, we assume the hypotheses of point (1) of \cref{thm:main}.
\cref{sub:left:order,,sub:bi:order} feature proofs of the two points of \cref{thm:main}.

\begin{remark}
Since $\Phi$ is a groupoid functor, the following facts are readily checked:
\begin{itemize}
    \item if $g\in\mathrm{C}\G(x)$ is not an identity, then: $g \in P_{\mathrm{C}\G(x)} \iff g^{-1} \not\in P_{\mathrm{C}\G(x)}$;
    \item the identities of $\mathrm{C}\G(x)$ do not belong to $P_{\mathrm{C}\G(x)}$;
    \item $g,h \in P_{\mathrm{C}\G(x)} \implies gh \in P_{\mathrm{C}\G(x)}$.
\end{itemize}
\end{remark}

\subsection{Left-orderability}
\label{sub:left:order}

To start, we need the following fact.

\begin{lemma}
\label{lem:normal:forms}
If $TgS^{-1} = T'hS'^{-1}$, then $g$ and $h$ have the same sign in $\G$ (i.e., either both $g$ and $h$ are positive, both are negative or both are identities).
\end{lemma}

\begin{proof}
The identity in $\Pi(\F\bowtie\G)$ is $TgS^{-1} (T'hS'^{-1})^{-1} = TgS^{-1} S'h^{-1}T'^{-1}$.
The factor in $\G$ is $g^B (h^A)^{-1}$ by \cref{rmk:composition}, so $g^B (h^A)^{-1}$ must be an identity.
Let us assume for a contradiction that $g$ and $h^{-1}$ are positive.
Then, by compatibility of the $\Phi$-discrete ordering, $(h^A)^{-1}$ is positive, so $g^B (h^A)^{-1}$ is positive.
Then the identity $g^B (h^A)^{-1}$ is positive, which is a contradiction.
\end{proof}

Consider the following subset of $\Pi(\F\bowtie\G)$, where $<_\F$ denotes the left-ordering of $\Pi(\F,\varepsilon)$:
\begin{equation}
\tag{$\star$}
\label{eq:cone}
P \coloneq \{ T g S^{-1} \mid g \text{ is positive} \} \cup \{ TS^{-1} \mid 1 <_\F TS^{-1} \}.
\end{equation}
Let us show that $P$ is a positive cone.

(1)
Consider two elements $TgS^{-1}$ and $T'hS'^{-1}$ of $P$.
If $g$ and $h$ are trivial, then $1 <_\F TS^{-1}$ and $1 <_\F T'S'^{-1}$, so $1 <_\F TS^{-1}T'S'^{-1}$, as needed.
Else, by \cref{rmk:composition} the factor in $\G$ of the product of $TgS^{-1}$ and $T'hS'^{-1}$ is $g^B [(h^{-1})^A]^{-1}$.
If $g$ and $h$ are positive, by compatibility we have that $g^B$ and $[(h^{-1})^A]^{-1}$ are positive, so $g^B [(h^{-1})^A]^{-1}$ is positive, as needed.
If one is trivial and the other is positive,
the product of $g^B$ and $[(h^{-1})^A]^{-1}$ (one of which factor is trivial) is positive.

(2)
Assume that an element belongs to both $P$ and $P^{-1}$.
Since it belongs to $P$, it can be written as $TgS^{-1}$ with $g$ positive or as $TS^{-1} >_\F 1$.
On the other hand, since it belongs to $P^{-1}$ it can be written as $T'hS'^{-1}$ with $h$ negative or as $T'S'^{-1} <_\F 1$.
By \cref{lem:normal:forms}, this is impossible.

(3)
By \cref{rmk:normal:forms}, every element of $\Pi(\F\bowtie\G)$ can be written as some $TgS^{-1}$.
If $g=1$, then $TgS^{-1}=TS^{-1} \neq 1$ belongs to $P$ or $P^{-1}$, depending on whether $1 <_\F TS^{-1}$ or $TS^{-1} <_\F 1$.
If $g \neq 1$, then $TgS^{-1}$ belongs to $P$ or $P^{-1}$, depending on whether $g$ is positive or negative.
Ultimately, $TgS^{-1}\in P \cup P^{-1}$, as needed.

Then $\Pi(\F\bowtie\G)$ is left-orderable, which is the first point of \cref{thm:main}.

\subsection{Bi-orderability}
\label{sub:bi:order}

For an indirect product $\F\bowtie\G$ with $\G$ discrete, consider the relation on $\G$ which relates elements of $\G(y,y)$ with the same action on $\F(y,-)$:
$$ g \sim h \ \iff \ \exists y\in\Ob(\G) \mid g,h\in\G(y,y) \ \wedge \ \forall f \in \F(y,-),\ g \cdot f = h \cdot f, $$

The relation $\sim$ is a congruence relation on $\G$.
Indeed, given $g_1 \sim g_2$ and $h_1 \sim h_2$, let us show that $g_1h_1 \sim g_2h_2$.
First note that related elements must have the same source and target, so $g_1,g_2\in\G(y,y)$ and $h_1,h_2\in\G(y,y)$ for some $y\in\Ob(\F\bowtie\G)$.
For any $f\in\F(y,-)$, we can readily compute the desired identity:
$$ (g_1h_1) \cdot f = g_1 \cdot (h_1 \cdot f) = g_1 \cdot (h_2 \cdot f) = g_2 \cdot (h_2 \cdot f) = g_2h_2 \cdot f. $$
This allows us to consider the quotient category $\G^* \coloneq \G/\sim$.

The projection $\G \to \G^*,\ g \mapsto g^*$ induces a surjective group homomorphism
$$\Psi \colon \Pi(\F\bowtie\G) \to \Pi(\F\bowtie\G^*),\ TgS^{-1} \mapsto Tg^*S^{-1}.$$
The category $\F \bowtie \G^*$ needs to be defined as an external indirect product, but it can naturally be treated as an internal one by \cite[Lemma 4.1]{Witzel}.

\begin{remark}
\label{rmk:pure}
Observe that, by its very \cref{def:pure}, the indirect product is pure if and only if $\G^*$ is trivial, or equivalently $\Psi\left(\Pi(\F\bowtie\G)\right) = \Pi(\F)$.
\end{remark}

Let $K=\mathrm{Ker}(\Psi)$, which is the subgroup of $\Pi(\F\bowtie\G)$ consisting of all the products $fgf^{-1}$ such that $g$ acts trivially on $\F$.

\begin{lemma}
\label{lem:semidirect}
If the indirect product between $\F$ and $\G$ is pure, then the group $\Pi(\F\bowtie\G)$ decomposes as a semidirect product $K \rtimes \Pi(\F)$.
\end{lemma}

\begin{proof}
Since the indirect product is pure, the following is a short exact sequence
$$K \xhookrightarrow{\mathrm{i}} \Pi(\F\bowtie\G) \xtwoheadrightarrow{\Psi} \Pi(\F)$$
by \cref{rmk:pure}, where $\mathrm{i}$ denotes the inclusion of $K$ into $\Pi(\F\bowtie\G)$.
The map $\Pi(\F) \rightarrow \Pi(\F \bowtie \G)$ sending each $TS^{-1}$ to itself is a section, so we are done.
\end{proof}

\begin{proposition}
\label{prop:K:order}
If $\G$ has a $\Phi$-discrete bi-ordering that is compatible with $\F$, then $K$ inherits a bi-ordering from that of $\G$.
\end{proposition}

\begin{proof}
Let $P$ be the positive cone for $\Pi(\F\bowtie\G)$ defined in \cref{eq:cone} and let
$$P_K \coloneq P \cap K = \{ ThT^{-1} \mid h \text{ is positive and has trivial action} \}.$$
This is a positive cone for $K$ and, to see that it defines a bi-order, we need to show that it is closed under conjugacy by $K$.
Let us check this:
\begin{gather*}
(fgf^{-1}) (ThT^{-1}) (fgf^{-1})^{-1} = \{f (g \cdot B) g^B [(h^{-1})^A]^{-1} (h^{-1} \cdot A)^{-1} T^{-1}\} (f g^{-1} f^{-1})= \\
= \{ fB g^B [(h^{-1})^A]^{-1} A^{-1}T^{-1} \} (fg^{-1}f^{-1}) =\\
= fB (g^B [(h^{-1})^A]^{-1} \cdot D) \{g^B [(h^{-1})^A]^{-1}\}^D (g^C)^{-1} (g \cdot C)^{-1} f^{-1},
\end{gather*}
where $A,B,C,D\in\F$ are such that $TA=fB$ and $fC=TAD$.
The factor in $\G$ is $\Delta \coloneq \{g^B [(h^{-1})^A]^{-1}\}^D (g^C)^{-1} = g^{BD} \{[(h^{-1})^A]^{-1}\}^D (g^C)^{-1}$.
Since $C=BD$, this $\Delta$ is a conjugate of $\{[(h^{-1})^A]^{-1}\}^D$, which is positive by compatibility (\cref{def:compatible:ordering}).
The image of $\Delta$ via $\Phi$ is thus the conjugate of a positive element in the bi-ordered image group, hence it is positive.
Then the product belongs to $P$, and it belongs to $K$ because it is a product of elements of $K$, thus it lies in $P_K$.
\end{proof}

It is known that the semidirect product $G_1 \rtimes G_2$ of bi-ordered groups $G_1$ and $G_2$ is bi-ordered as soon as the action of $G_2$ by conjugacy on $G_1$ is order-preserving (for example \cite[Problem 1.23]{OrderableGroups}).
In our case, $\Pi(\F)$ is bi-ordered by hypothesis, $K$ is bi-ordered by \cref{prop:K:order}, and conjugating an element of $K$ by one of $\Pi(\F)$ does not change the factor in $\G$, so $K \rtimes \Pi(\F)$ is bi-ordered.
By \cref{lem:semidirect}, this is $\Pi(\F\bowtie\G)$, so we are done.


\section{Groups arising from digit rewriting systems}
\label{sec:edge:shifts}

In this section we introduce groups of $\G$-fractions of digit rewriting systems:
these can be thought of as generalizations of certain diagram groups, where $\G$ is some groupoid that permutes the digits.
We will use braided fractions as examples of applications of \cref{thm:main}.

\subsection{The Ore category of a digit rewriting system}

Consider a rewriting system whose set of objects is the set $\Sigma^*$ of finite words over an alphabet $\Sigma$ and whose set of rewriting rules $\R$ are as follows and there is at most one for each $x \in \Sigma$:
$$ x \leftarrow u_1 \dots u_k \ (k \geq 2, u_i\in\Sigma). $$
We call these \textbf{digit rewriting systems}, or \textbf{DRSs} for short.
\cref{sub:G:fractions} features some examples.

This determines a category $\F$ in the following way.
The set of objects is the same as that of the rewriting system, i.e., finite words over the alphabet $\Sigma$.
Morphisms are generated by applications of rewritings rules of $\R$, i.e.,
$$ w_1 x w_2 \leftarrow w_1 u_1 \dots u_k w_2 $$
where $w_1$ and $w_2$ are finite (possibly empty) words and $x \leftarrow u_1 \dots u_k$ belongs to $\R$.
The order in which two disjoint rewritings are performed does not matter, meaning that this category is subject to the relations
\begin{gather*}
w_1 x_1 w_2 x_2 w_3 \leftarrow w_1 x_1 w_2 v_2 w_3 \leftarrow w_1 v_1 w_2 v_2 w_3 =\\
= w_1 x_1 w_2 x_2 w_3 \leftarrow w_1 v_1 w_2 x_2 w_3 \leftarrow w_1 v_1 w_2 v_2 w_3.
\end{gather*}
(These relations are natural generalizations of those that are described, for example, in \cite[Proposition 2.3, Remark 2.4]{Witzel}.)

It is easy to see that $\F$ is cancellative and has common right multiples, so $\F$ is right-Ore.
Since rewritings decrease lengths, $\F$ has no non-trivial invertibles.

\begin{remark}
\label{rmk:F:diagram:group:orderable}
A quick comparison reveals that $\F$ is the $1$-skeleton of the Squier complex of the semigroup presentation corresponding to the rewriting system, so $\Pi(\F)$ is a diagram group (see \cite[Theorem 6.1]{DiagramGroups}).
This implies that it is bi-orderable by \cite[Theorem 6.1]{DiagramGroupsOrderable}.
\end{remark}

\subsection{Groups of \texorpdfstring{$\G$}{G}-fractions of a DRS}
\label{sub:G:fractions}

Consider an indirect product $\F\bowtie\G$, where $\F$ is a DRS.
Motivated by \cref{rmk:normal:forms}, we refer to a fundamental group $\Pi(\F\bowtie\G)$ as a \textbf{group of $\G$-fractions of $\F$}.

To illustrate what this is, let us first consider the groupoid whose morphisms are
$$ x_{\sigma(1)} \dots x_{\sigma(m)} \leftarrow x_1 \dots x_m, $$
for all $\sigma \in \mathrm{Sym}\{1,\dots,m\}$ and $m \in \mathbb{N}$.
We denote this category by $\P$ (for \textit{permutation}), although it implicitly depends on the objects of $\F$.
The category generated by $\F$ and $\P$ is an indirect product thanks to the following actions of $\F$ and $\G$ on each other:
if $f\in\F(w,-)$ is a single rewriting and $g\in\G(-,w)$, then
\begin{gather*}
f = ( w = x_1 \dots x_m \leftarrow x_1 \dots x_{l-1} u_1 \dots u_k x_{l+1} \dots x_m ),\\
g = ( x_{\sigma(1)} \dots x_{\sigma(m)} \leftarrow x_1 \dots x_m ).
\end{gather*}
We represent these with diagrams like those depicted in \cref{fig:f:g}, which will surely be familiar to a reader who is familiar with Thompson's groups and their generalizations.
Checking that the actions of $\F$ and $\G$ on each other satisfy the axioms of indirect products (\cite[Subsection 4.2]{Witzel}) is standard routine.

\begin{figure}\centering
\begin{subfigure}[b]{\textwidth}
\centering
\begin{tikzpicture}[yscale=.8]
    \node at (-5, -1) {$f\ =$};
    \node at (-4, -3.5) {$g\ =$};
    \node (A1) at (-4,0) {$x_1$};
    \node (A2) at (-3,0) {$x_2$};
    \node (A3) at (-2,0) {$\dots$};
    \node[blue] (Al1) at (-1,0) {$u_1$};
    \node[blue] (Al2) at (0,0) {$\dots$};
    \node[blue] (Al3) at (1,0) {$u_k$};
    \node (A4) at (2,0) {$\dots$};
    \node (A5) at (3,0) {$\dots$};
    \node (A6) at (4,0) {$x_{m-1}$};
    \node (A7) at (5,0) {$x_m$};
    \begin{scope}[yshift=-2.5cm]
    \node (B1) at (-3,0) {$x_1$};
    \node (B2) at (-2,0) {$x_2$};
    \node (B3) at (-1,0) {$\dots$};
    \node[red] (Bl) at (0,0) {$x_l$};
    \node (B4) at (1,0) {$\dots$};
    \node (B5) at (2,0) {$\dots$};
    \node (B6) at (3,0) {$x_{m-1}$};
    \node (B7) at (4,0) {$x_m$};
    \end{scope}
    \begin{scope}[yshift=-5cm]
    \node (C1) at (-3,0) {$x_{\overline{1}}$};
    \node (C2) at (-2,0) {$x_{\overline{2}}$};
    \node (C3) at (-1,0) {$\dots$};
    \node (C4) at (0,0) {$\dots$};
    \node[red] (Cl) at (1,0) {$x_l$};
    \node (C5) at (2,0) {$\dots$};
    \node (C6) at (3,0) {$x_{\overline{m-1}}$};
    \node (C7) at (4,0) {$x_{\overline{m}}$};
    \end{scope}
    \node[circle,fill,inner sep=1pt] (AB) at ($.5*(Al2)+.5*(Bl)$) {};
    \draw (A1) to (B1);
    \draw (A2) to (B2);
    \draw[dashed] (A3) to (B3);
    \draw[dashed] (A4) to (B4);
    \draw[dashed] (A5) to (B5);
    \draw (A6) to (B6);
    \draw (A7) to (B7);
    \draw[blue] (Al1) to (AB);
    \draw[dashed,blue] (Al2) to (AB);
    \draw[blue] (Al3) to (AB);
    \draw[red] (AB) to (Bl);
    \draw[dashed] (B1) to ($.5*(C3)+.5*(C4)+(0,.2)$);
    \draw (B2) to (C1);
    \draw[dashed] (B3) to (C6);
    \draw[dashed] ($.5*(B4)+.5*(B5)-(0,0.2)$) to (C7);
    \draw (B6) to (C2);
    \draw[dashed] (B7) to (C5);
    \draw[red,curvy] (Bl) to (Cl);
\end{tikzpicture}
\caption{Schematic diagrams for elements $f\in\F$ and $g\in\P$ and their product $gf$.}
\label{fig:f:g:diagrams}
\end{subfigure}
\\\vspace{\baselineskip}
\begin{subfigure}[b]{\textwidth}
\centering
\begin{tikzpicture}[yscale=.8]
    \node at (-5, -1) {$g^f\ =$};
    \node at (-5, -3.5) {$g \cdot f\ =$};
    \node (A1) at (-4,0) {$x_1$};
    \node (A2) at (-3,0) {$x_2$};
    \node (A3) at (-2,0) {$\dots$};
    \node[blue] (Al1) at (-1,0) {$u_1$};
    \node[blue] (Al2) at (0,0) {$\dots$};
    \node[blue] (Al3) at (1,0) {$u_k$};
    \node (A4) at (2,0) {$\dots$};
    \node (A5) at (3,0) {$\dots$};
    \node (A6) at (4,0) {$x_{m-1}$};
    \node (A7) at (5,0) {$x_m$};
    \begin{scope}[yshift=-2.5cm]
    \node (B1) at (-4,0) {$x_{\overline{1}}$};
    \node (B2) at (-3,0) {$x_{\overline{2}}$};
    \node (B3) at (-2,0) {$\dots$};
    \node (B4) at (-1,0) {$\dots$};
    \node[blue] (Bl1) at (0,0) {$u_1$};
    \node[blue] (Bl2) at (1,0) {$\dots$};
    \node[blue] (Bl3) at (2,0) {$u_k$};
    \node (B5) at (3,0) {$\dots$};
    \node (B6) at (4,0) {$x_{\overline{m-1}}$};
    \node (B7) at (5,0) {$x_{\overline{m}}$};
    \end{scope}
    \begin{scope}[yshift=-5cm]
    \node (C1) at (-3,0) {$x_{\overline{1}}$};
    \node (C2) at (-2,0) {$x_{\overline{2}}$};
    \node (C3) at (-1,0) {$\dots$};
    \node (C4) at (0,0) {$\dots$};
    \node[red] (Cl) at (1,0) {$x_l$};
    \node (C5) at (2,0) {$\dots$};
    \node (C6) at (3,0) {$x_{\overline{m-1}}$};
    \node (C7) at (4,0) {$x_{\overline{m}}$};
    \end{scope}
    \node[circle,fill,inner sep=1pt] (BC) at ($.5*(Bl2)+.5*(Cl)$) {};
    \draw[dashed] (A1) to ($.5*(B3)+.5*(B4)+(0,.2)$);
    \draw (A2) to (B1);
    \draw[dashed] (A3) to (B6);
    \draw[dashed] ($.5*(A4)+.5*(A5)-(0,.2)$) to (B7);
    \draw (A6) to (B2);
    \draw[dashed] (A7) to (B5);
    \draw[curvy,blue] (Al1) to (Bl1);
    \draw[dashed,curvy,blue] (Al2) to (Bl2);
    \draw[curvy,blue] (Al3) to (Bl3);
    \draw (B1) to (C1);
    \draw (B2) to (C2);
    \draw[dashed] (B3) to (C3);
    \draw[dashed] (B4) to (C4);
    \draw[dashed] (B5) to (C5);
    \draw (B6) to (C6);
    \draw (B7) to (C7);
    \draw[blue] (Bl1) to (BC);
    \draw[dashed,blue] (Bl2) to (BC);
    \draw[blue] (Bl3) to (BC);
    \draw[red] (BC) to (Cl);
\end{tikzpicture}
\caption{The actions of $\F$ and $\P$ on each other and their product $(g \cdot f) g^f$.}
\label{fig:f:g:action}
\end{subfigure}
\caption{Diagrams for $\Pi(\F\bowtie\P)$, where $\overline{i}$ denotes $\sigma^{-1}(i)$.}
\label{fig:f:g}
\end{figure}

Groups of $\P$-fractions of $\F$ are not orderable since they have torsion, but in \cref{sub:braided:fractions} we introduce braided versions to which \cref{thm:main} applies.

\subsubsection{The Higman-Thompson groups}

Consider $\Sigma=\{x\}$ and the rewriting rule
$$ x \leftarrow x^n. $$
This defines a category $\F_n$.
The fundamental group $\Pi(\F_n,k)$ based at some $x^k\in\Ob(\F_n)$ is the Higman-Thompson group $F_{n,k}$ \cite{HigmanThompson}.
In particular, for $n=2$ and $k=1$ this is Thompson's group $F$.
Similarly, $\Pi(\F_n\bowtie\P,x^k)$ is the Higman-Thompson group $V_{n,k}$ and the Higman-Thompson groups $T_{n,k}$ can be built by replacing $\P$ with a category of cyclic digit permutations.

\subsubsection{The Houghton groups}
\label{sub:houghton}

Consider $\Sigma=\{x,y_1,\dots,y_n\}$ and the rewriting rules
$$ y_i \leftarrow y_i x,\ \forall i \in \{1, \dots, n\}. $$
This defines a category $\FH_n$.
The fundamental group $\Pi(\FH_n \bowtie \P, y_1 \dots y_n)$ is the Houghton group $H_n$ \cite{Houghton}.

\subsubsection{Topological full groups of edge shifts}

For a directed graph $\Gamma$, the edge shift $E\Gamma$ (see for example \cite[\S2.2]{LindMarcus}) defines a DRS, once an ordering of each set of edges of $\Gamma$ originating at a vertex is fixed.
The rewriting rules are then
$$ v \leftarrow w_1 \dots w_m, $$
where the $w_i$'s are the terminal vertices of the edges originating from $v$.
This defines a category $\F_\Gamma$.
The fundamental group $\Pi(\F_\Gamma \bowtie \P,u)$ is the topological full group of the initial edge shift $E\Gamma$ starting at $u$ \cite{Matui} (see \cite{Waltraud} for an interpretation of these groups as groups of almost automorphisms of trees).

\subsection{Groups of braided fractions of DRSs}
\label{sub:braided:fractions}

\subsubsection{Groupoid of digital braids}

Let $\R$ be a DRS with alphabet $\Sigma$.
Let us consider a groupoid $\B_{\R}$ (or $\B$ when $\R$ is understood) whose set of objects is $\Sigma^*$ and whose morphisms are digital braids.
Given words $w_1 = x_1 \dots x_l$ and $w_2 = y_1 \dots y_l$ in which every digit of the alphabet appears the same amount of times, an element of $\B(w_2,w_1)$ consists of $l$ disjoint arcs running monotonically inside a cylinder, from top to bottom, where the top and bottom of the cylinder are equipped with a tuple $(p_1, \dots, p_{l})$ of points labeled by the digits of $w_1$ and $w_2$ in their order, respectively.
Arcs are required to join points labeled by the same digit.
Two such morphisms are defined up to isotopy and composition is defined by gluing the top of one with the bottom of the other when the labels coincide.
For example, the left side of \cref{fig:digital:braid} portrays a morphism and \cref{fig:digital:braid:composition} shows and a composition.

\begin{figure}\centering
\begin{tikzpicture}
    \draw[gray] (1.5,-2) ellipse (2.2cm and .6cm); 
    \draw[line width=4pt,white] (3,0) to[out=-90,in=90] (0,-2);
    \draw[red] (1,0) to[out=-90,in=90] (0,-2);
    \draw[line width=4pt,white] (3,0) to[out=-90,in=90] (2,-2);
    \draw[ForestGreen] (3,0) to[out=-90,in=90] (2,-2);
    \draw[line width=4pt,white] (0,0) to [out=-90,in=90] (3,-2);
    \draw[red] (0,0) to[out=-90,in=90] (3,-2);
    \draw[line width=4pt,white] (2,0) to [out=-90,in=90] (1,-2);
    \draw[blue] (2,0) to[out=-90,in=90] (1,-2);
    \draw[line width=4pt,white] (1.5,0) ellipse (2.2cm and .6cm); 
    \draw[gray] (-.7,0) to (-.7,-2) (3.7,0) to (3.7,-2);
    \draw[gray] (1.5,0) ellipse (2.2cm and .6cm); 
    \draw[red] (0,0) node[circle,inner sep=1.5pt,fill]{} node[above]{$x$};
    \draw[red] (1,0) node[circle,inner sep=1.5pt,fill]{} node[above]{$x$};
    \draw[blue] (2,0) node[circle,inner sep=1.5pt,fill]{} node[above]{$y$};
    \draw[ForestGreen] (3,0) node[circle,inner sep=1.5pt,fill]{} node[above]{$z$};
    \draw[red] (0,-2) node[circle,inner sep=1.5pt,fill]{} node[below]{$x$};
    \draw[blue] (1,-2) node[circle,inner sep=1.5pt,fill]{} node[below]{$y$};
    \draw[ForestGreen] (2,-2) node[circle,inner sep=1.5pt,fill]{} node[below]{$z$};
    \draw[red] (3,-2) node[circle,inner sep=1.5pt,fill]{} node[below]{$x$};
    \node at (4.5,-1) {$\xrightarrow{\Phi}$};
    \begin{scope}[xshift=6cm]
    \draw[gray] (1.5,-2) ellipse (2.2cm and .6cm); 
    \draw[line width=4pt,white] (1,0) to[out=-90,in=90] (0,-2);
    \draw (1,0) to[out=-90,in=90] (0,-2);
    \draw[line width=4pt,white] (3,0) to[out=-90,in=90] (2,-2);
    \draw (3,0) to[out=-90,in=90] (2,-2);
    \draw[line width=4pt,white] (0,0) to [out=-90,in=90] (3,-2);
    \draw (0,0) to[out=-90,in=90] (3,-2);
    \draw[line width=4pt,white] (2,0) to [out=-90,in=90] (1,-2);
    \draw (2,0) to[out=-90,in=90] (1,-2);
    \draw[line width=4pt,white] (1.5,0) ellipse (2.2cm and .6cm); 
    \draw[gray] (-.7,0) to (-.7,-2) (3.7,0) to (3.7,-2);
    \draw[gray] (1.5,0) ellipse (2.2cm and .6cm); 
    \draw (0,0) node[circle,inner sep=1.5pt,fill]{};
    \draw (1,0) node[circle,inner sep=1.5pt,fill]{};
    \draw (2,0) node[circle,inner sep=1.5pt,fill]{};
    \draw (3,0) node[circle,inner sep=1.5pt,fill]{};
    \draw (0,-2) node[circle,inner sep=1.5pt,fill]{};
    \draw (1,-2) node[circle,inner sep=1.5pt,fill]{};
    \draw (2,-2) node[circle,inner sep=1.5pt,fill]{};
    \draw (3,-2) node[circle,inner sep=1.5pt,fill]{};
    \end{scope}
\end{tikzpicture}
\caption{A morphism of $\B(\textcolor{red}{x}\textcolor{blue}{y}\textcolor{ForestGreen}{z}\textcolor{red}{x}, \textcolor{red}{x}\textcolor{red}{x}\textcolor{blue}{y}\textcolor{ForestGreen}{z})$ (colors for visual aid) and its image in $\mathrm{B}_4$ via the digit-forgetting functor $\Phi$.}
\label{fig:digital:braid}
\end{figure}

\begin{figure}\centering
\begin{tikzpicture}
    \draw[gray] (1.5,-2) ellipse (2.2cm and .6cm); 
    \draw[line width=4pt,white] (1,0) to[out=-90,in=90] (0,-2);
    \draw[red] (1,0) to[out=-90,in=90] (0,-2);
    \draw[line width=4pt,white] (3,0) to[out=-90,in=90] (2,-2);
    \draw[ForestGreen] (3,0) to[out=-90,in=90] (2,-2);
    \draw[line width=4pt,white] (0,0) to [out=-90,in=90] (3,-2);
    \draw[red] (0,0) to[out=-90,in=90] (3,-2);
    \draw[line width=4pt,white] (2,0) to [out=-90,in=90] (1,-2);
    \draw[blue] (2,0) to[out=-90,in=90] (1,-2);
    \draw[line width=4pt,white] (1.5,0) ellipse (2.2cm and .6cm); 
    \draw[gray] (-.7,0) to (-.7,-2) (3.7,0) to (3.7,-2);
    \draw[gray] (1.5,0) ellipse (2.2cm and .6cm); 
    \draw[red] (0,0) node[circle,inner sep=1.5pt,fill]{} node[above]{$x$};
    \draw[red] (1,0) node[circle,inner sep=1.5pt,fill]{} node[above]{$x$};
    \draw[blue] (2,0) node[circle,inner sep=1.5pt,fill]{} node[above]{$y$};
    \draw[ForestGreen] (3,0) node[circle,inner sep=1.5pt,fill]{} node[above]{$z$};
    \begin{scope}[yshift=-2cm]
    \draw[gray] (1.5,-2) ellipse (2.2cm and .6cm); 
    \draw[line width=4pt,white] (1,0) to[out=-90,in=90] (2,-2);
    \draw[blue] (1,0) to[out=-90,in=90] (2,-2);
    \draw[line width=4pt,white] (2,0) to[out=-90,in=90] (3,-2);
    \draw[ForestGreen] (2,0) to[out=-90,in=90] (3,-2);
    \draw[line width=4pt,white] (3,0) to[out=-90,in=90] (0,-2);
    \draw[red] (3,0) to[out=-90,in=90] (0,-2);
    \draw[line width=4pt,white] (0,0) to[out=-90,in=90] (1,-2);
    \draw[red] (0,0) to[out=-90,in=90] (1,-2);
    \draw[line width=4pt,white] (-.7,0) arc(180:360:2.2cm and .6cm); 
    \draw[gray] (-.7,0) to (-.7,-2) (3.7,0) to (3.7,-2);
    \draw[gray] (-.7,0) arc(180:360:2.2cm and .6cm); 
    \draw[red] (0,0) node[circle,inner sep=1.5pt,fill]{} node[left]{$x$};
    \draw[blue] (1,0) node[circle,inner sep=1.5pt,fill]{} node[left]{$y$};
    \draw[ForestGreen] (2,0) node[circle,inner sep=1.5pt,fill]{} node[left]{$z$};
    \draw[red] (3,0) node[circle,inner sep=1.5pt,fill]{} node[left]{$x$};
    \draw[red] (0,-2) node[circle,inner sep=1.5pt,fill]{} node[below]{$x$};
    \draw[red] (1,-2) node[circle,inner sep=1.5pt,fill]{} node[below]{$x$};
    \draw[blue] (2,-2) node[circle,inner sep=1.5pt,fill]{} node[below]{$y$};
    \draw[ForestGreen] (3,-2) node[circle,inner sep=1.5pt,fill]{} node[below]{$z$};
    \end{scope}
    \node at (-1,-1) {$g_1$};
    \node at (-1,-3) {$g_2$};
    \node at (4.5,-2) {$=$};
    \node at (10.25,-2) {$g_2 g_1$};
    \begin{scope}[xshift=6cm,yshift=-.25cm]
    \draw[gray] (1.5,-3.5) ellipse (2.2cm and .6cm); 
    \draw[line width=4pt,white] (3,0) to[out=-90,in=90] (3,-3.5);
    \draw[ForestGreen] (3,0) to[out=-90,in=90] (3,-3.5);
    \draw[line width=4pt,white] (1,0) to[out=-90,in=90,looseness=1.25] (0,-1.75);
    \draw[red] (1,0) to[out=-90,in=90,looseness=1.25] (0,-1.75);
    \draw[line width=4pt,white] (0,0) to [out=-90,in=90,looseness=1.25] (2,-1.75);
    \draw[red] (0,0) to[out=-90,in=90,looseness=1.25] (2,-1.75);
    \draw[line width=4pt,white] (2,0) to[out=-90,in=90] (1,-1.75);
    \draw[blue] (2,0) to[out=-90,in=90] (1,-1.75);
    \draw[line width=4pt,white] (1,-1.75) to[out=-90,in=90] (2,-3.5);
    \draw[blue] (1,-1.75) to[out=-90,in=90] (2,-3.5);
    \draw[line width=4pt,white] (2,-1.75) to [out=-90,in=90,looseness=1.25] (0,-3.5);
    \draw[red] (2,-1.75) to [out=-90,in=90,looseness=1.25] (0,-3.5);
    \draw[line width=4pt,white] (0,-1.75) to [out=-90,in=90,looseness=1.25] (1,-3.5);
    \draw[red] (0,-1.75) to [out=-90,in=90,looseness=1.25] (1,-3.5);
    \draw[line width=4pt,white] (1.5,0) ellipse (2.2cm and .6cm); 
    \draw[gray] (-.7,0) to (-.7,-3.5) (3.7,0) to (3.7,-3.5);
    \draw[gray] (1.5,0) ellipse (2.2cm and .6cm); 
    \draw[red] (0,0) node[circle,inner sep=1.5pt,fill]{} node[above]{$x$};
    \draw[red] (1,0) node[circle,inner sep=1.5pt,fill]{} node[above]{$x$};
    \draw[blue] (2,0) node[circle,inner sep=1.5pt,fill]{} node[above]{$y$};
    \draw[ForestGreen] (3,0) node[circle,inner sep=1.5pt,fill]{} node[above]{$z$};
    \draw[red] (0,-3.5) node[circle,inner sep=1.5pt,fill]{} node[below]{$x$};
    \draw[red] (1,-3.5) node[circle,inner sep=1.5pt,fill]{} node[below]{$x$};
    \draw[blue] (2,-3.5) node[circle,inner sep=1.5pt,fill]{} node[below]{$y$};
    \draw[ForestGreen] (3,-3.5) node[circle,inner sep=1.5pt,fill]{} node[below]{$z$};
    \end{scope}
\end{tikzpicture}
\caption{A composition $g_2 g_1$ between a $g_1 \in \B(\textcolor{red}{x}\textcolor{blue}{y}\textcolor{ForestGreen}{z}\textcolor{red}{x}, \textcolor{red}{x}\textcolor{red}{x}\textcolor{blue}{y}\textcolor{ForestGreen}{z})$ and a $g_2 \in \B( \textcolor{red}{x}\textcolor{red}{x}\textcolor{blue}{y}\textcolor{ForestGreen}{z},\textcolor{red}{x}\textcolor{blue}{y}\textcolor{ForestGreen}{z}\textcolor{red}{x})$. The composition is a pure digital braid.}
\label{fig:digital:braid:composition}
\end{figure}

\subsubsection{Braided fractions of DRSs}

The category generated by $\F$ and $\B$ is an indirect product $\F\bowtie\B$ (as with the permutation groupoid $\P$, checking this is routine).
Moreover, the action of $\F$ on $\B$ is injective (if the same expansion of two braids is the same, then the two braids were the same to begin with).
Thus, by \cref{lem:product}, $\F\bowtie\B$ is right-Ore.
The group of $\B$-fractions of $\F$, which is a braided version of the topological full group of an edge shift, thus belongs to the framework of \cite{Witzel}.

As an example, \cref{fig:new:braided:houghton} depicts an element of $\Pi(\FH_3\bowtie\B, y_1 y_2 y_3)$, where $\FH_3$ is the category described in \cref{sub:houghton}.
These groups include the braided Houghton groups, as we will see in \cref{sub:left:orderable:braided:houghton}.

\begin{figure}\centering
\begin{tikzpicture}
    \draw[line width=4pt,white] (.5,-3.5) to[out=90,in=-90] (0,-2);
    \draw[line width=4pt,white] (2,-3.5) to (2,-3) to (1,-2);
    \draw[line width=4pt,white] (3.5,-3.5) to[out=90,in=-90] (4,-2);
    \draw[line width=4pt,white] (2,-3) to (3,-2) (1.5,-2.5) to (2,-2);
    \draw[red] (.5,-3.5) to[out=90,in=-90] (0,-2);
    \draw[blue] (2,-3.5) to (2,-3) to (1,-2);
    \draw[ForestGreen] (3.5,-3.5) to[out=90,in=-90] (4,-2);
    \draw[black] (2,-3) to (3,-2) (1.5,-2.5) to (2,-2);
    \draw[line width=4pt,white] (2,-2) ellipse (2.7cm and .65cm); 
    \draw[gray] (2,-2) ellipse (2.7cm and .65cm); 
    \draw[line width=4pt,white] (0,0) to[out=-90,in=90] (0,-2);
    \draw[red] (0,0) to[out=-90,in=90] (0,-2);
    \draw[line width=4pt,white] (2,0) to[out=-90,in=90] (1,-2);
    \draw[blue] (2,0) to[out=-90,in=90] (1,-2);
    \draw[line width=4pt,white] (1,0) to [out=-90,in=90] (3,-2);
    \draw[black] (1,0) to[out=-90,in=90] (3,-2);
    \draw[line width=4pt,white] (4,0) to [out=-90,in=90] (2,-2);
    \draw[black] (4,0) to[out=-90,in=90] (2,-2);
    \draw[line width=4pt,white] (3,0) to [out=-90,in=90] (4,-2);
    \draw[ForestGreen] (3,0) to[out=-90,in=90] (4,-2);
    \draw[line width=4pt,white] (2,0) ellipse (2.7cm and .65cm); 
    \draw[gray] (-.7,0) to (-.7,-2) (4.7,0) to (4.7,-2);
    \draw[gray] (2,0) ellipse (2.7cm and .65cm); 
    \draw[line width=4pt,white] (.5,1.5) to (.5,.75) to (0,0);
    \draw[line width=4pt,white] (2,1.5) to (2,0);
    \draw[line width=4pt,white] (3.5,1.5) to (3.5,.75) to (3,0);
    \draw[line width=4pt,white] (.5,.75) to (1,0) (3.5,.75) to (4,0);
    \draw[red] (.5,1.5) to (.5,.75) to (0,0);
    \draw[blue] (2,1.5) to (2,0);
    \draw[ForestGreen] (3.5,1.5) to (3.5,.75) to (3,0);
    \draw[black] (.5,.75) to (1,0) (3.5,.75) to (4,0);
    \draw[red] (.5,1.5) node[circle,inner sep=1.5pt,fill]{} node[left]{$y_1$};
    \draw[blue] (2,1.5) node[circle,inner sep=1.5pt,fill]{} node[left]{$y_2$};
    \draw[ForestGreen] (3.5,1.5) node[circle,inner sep=1.5pt,fill]{} node[left]{$y_3$};
    \draw[red] (0,0) node[circle,inner sep=1.5pt,fill]{} node[left]{$y_1$};
    \draw[black] (1,0) node[circle,inner sep=1.5pt,fill]{} node[left]{$x$};
    \draw[blue] (2,0) node[circle,inner sep=1.5pt,fill]{} node[left]{$y_2$};
    \draw[ForestGreen] (3,0) node[circle,inner sep=1.5pt,fill]{} node[left]{$y_3$};
    \draw[black] (4,0) node[circle,inner sep=1.5pt,fill]{} node[left]{$x$};
    \draw[red] (0,-2) node[circle,inner sep=1.5pt,fill]{} node[left]{$y_1$};
    \draw[blue] (1,-2) node[circle,inner sep=1.5pt,fill]{} node[left]{$y_2$};
    \draw[black] (2,-2) node[circle,inner sep=1.5pt,fill]{} node[left]{$x$};
    \draw[black] (3,-2) node[circle,inner sep=1.5pt,fill]{} node[left]{$x$};
    \draw[ForestGreen] (4,-2) node[circle,inner sep=1.5pt,fill]{} node[left]{$y_3$};
    \draw[red] (.5,-3.5) node[circle,inner sep=1.5pt,fill]{} node[left]{$y_1$};
    \draw[blue] (2,-3.5) node[circle,inner sep=1.5pt,fill]{} node[left]{$y_2$};
    \draw[ForestGreen] (3.5,-3.5) node[circle,inner sep=1.5pt,fill]{} node[left]{$y_3$};
\end{tikzpicture}
\caption{An element of $\Pi(\FH_3\bowtie\B, \textcolor{red}{y_1} \textcolor{blue}{y_2} \textcolor{ForestGreen}{y_3})$.}
\label{fig:new:braided:houghton}
\end{figure}

\subsubsection{Purely braided fractions of DRSs}

Let us consider the discrete subgroupoid $\PB$ of $\B$ consisting of the morphisms of $\B$ with trivial underlying digit permutation (in particular, the source and the target are the same).
For instance, the morphism depicted in \cref{fig:digital:braid} does not belong to $\PB$, while the morphism $g_2 g_1$ shown in \cref{fig:digital:braid:composition} does.

The category generated by $\F$ and $\PB$ is an indirect product $\F\bowtie\PB$.
The action of $\F$ on $\PB$, being the same as that on $\B$, is injective, so $\F\bowtie\PB$ is right-Ore.
Moreover, note that $\F\bowtie\PB$ is a pure indirect product (\cref{def:pure}).

\subsection{Orderability of the groups}
\label{sub:left:orderable:examples}

Let us finally show that the groups of (purely) braided fractions are (bi-) left-orderable.

\subsubsection{Left-orderability of groups of \texorpdfstring{$\B$}{B}-fractions of DRSs}

Let us denote by $\mathrm{B}_k$ the braid group on $k$ strands.
Let $\mathrm{B}$ be the discrete groupoid whose set of objects is $\mathbb{N}_{\geq 1}$ and where $\mathrm{B}(k,k) = \mathrm{B}_k$.
We endow $\mathrm{B}$ with the discrete left-ordering comprised of the Dehornoy left-orderings of braid groups (originally introduced in \cite{DehornoyOrder}; for a more clear definition see \cite[Chapter II, Section 1.2]{OrderingBraids}):
an element $g\in\mathrm{B}(k,k)$ is positive if and only if it is represented by a word in the Artin generators $\sigma_1, \dots, \sigma_{k-1}$ such that $\sigma_i$ occurs but $\sigma_1^{\pm1},\dots,\sigma_{i-1}^{\pm1},\sigma_i^{-1}$ do not, for some $i\in\{1,\dots,k-1\}$.
Let us call this property (D).

Given two words $w_1$ and $w_2$ of length $k$, we let $\Phi \colon \B(w_1,w_2) \to \mathrm{B}_k$ be the functor that forgets the digits, mapping a digital braid to the underlying braid.
See the right side of \cref{fig:digital:braid}, for example.
This determines a $\Phi$-discrete left-ordering on $\B$ induced by $\Phi$ (\cref{def:rho:discrete:ordering}).

\begin{remark}
Thanks to the axioms of indirect actions, it suffices to check compatibility (\cref{def:compatible:ordering}) on a set of generators of $\F$, by which we mean a set $S$ of morphisms of $\F$ such that every morphism of $\F$ is a product of elements of $S$.
\end{remark}

For a DRS $\F$ and for $\G=\B$, we can thus let $f\in\F$ be a single rewriting among those of the rewriting system $\R$ that defines $\F$.
Let us assume that $g\in\Phi^{-1}(\sigma_i) \in \B(-,w_1 x w_2)$, which is positive since $ 1_{\D(k,k)} \prec_k \Phi(g)$.
Since $f$ is a generator, let $f \in \F(w_1 x w_2, w_1 u w_2)$ be the rewriting of the instance of a word $w$ to the digit $x$.
Then the element $g^f$ is obtained from $\sigma_i$ by replacing the strand corresponding to $x$, say it is the $l$-th strand, with $k$ strands corresponding to $w$, similarly to \cref{fig:f:g} but also keeping track of isotopy.
For a generic positive $g$, writing it as a word satisfying property (D), the standard argument used for braided Thompson groups (see \cite[Subsection 4.1]{CloningSystemsOrderable} and, for more detail, \cite[Proposition 4.15]{DehornoyThompsonOrder}) shows that $g^f$ satisfies the same property (D), so it is positive.

This means that the $\Phi$-discrete left-ordering on $\B$ is compatible with $\F$.
Thus, by \cref{thm:main}, we have the following result.

\begin{corollary}
\label{cor:left:orderable}
$\Pi(\F\bowtie\B,w)$ is left-orderable, for any $w\in\Sigma^*$.
\end{corollary}

\subsubsection{Bi-orderability of groups of \texorpdfstring{$\PB$}{PB}-fractions of DRSs}

For pure braids, first recall that $\Pi(\F)$ is bi-orderable (\cref{rmk:F:diagram:group:orderable}) and that $\F\bowtie\PB$ is pure.
The groupoid $\PB$ is discrete and each $\PB(w,w)$ is a group of pure braids.
Such groups are bi-ordered with an ordering inherited from the Magnus ordering of free groups (this is explained in \cite[Subsection 2.2]{BraidedBiOrderable}, see \cite{MagnusBraid} for more detail).
We are thus only missing compatibility, which follows from \cite[Lemma 3.1]{BraidedBiOrderable} for the following reason.
That Lemma shows that, when replacing a strand of a pure braid with two strands, positivity is preserved.
When computing $g^f$ (for $f$ a single rewriting), a strand is replaced with $k$ of strands.
This is the same as replacing a strand with two strands $k-1$ times.
Ultimately then, positivity is preserved by the action of $\F$, so \cref{thm:main} implies the following result.

\begin{corollary}
$\Pi(\F\bowtie\PB,w)$ is bi-orderable, for any $w\in\Sigma^*$.
\end{corollary}

\subsection{Examples and applications}
\label{sub:left:orderable:braided:houghton}

\cref{cor:left:orderable} allows to immediately show that the braided versions of topological full groups of edge shifts are left-orderable.
Furthermore, the braided Houghton groups introduced in \cite{braidedHoughtonThesis} are left-orderable because they embed into groups of $\B$-fractions of the DRSs $\mathrm{R}$ described in \cref{sub:houghton}, as we explain right below.

Consider the wide subgroupoid $\BH$ of $\B_{\mathrm{R}}$ generated by the two types of morphisms depicted in \cref{fig:braided:houghton:generators}, where the $x$ and the $y_i$'s are as in \cref{sub:houghton}.
Intuitively, the morphisms of $\BH$ are the digital braids whose strands joining the digits $x$ can always be drawn on top of those joining digits $y_i$ (see for example \cref{fig:braided:houghton}).
Indeed, generators of type 2 cannot produce crossings between a strand joining the digits $x$ and a strand joining the digits $y_i$, so such crossings are necessarily the result of generators of type 1, where the strand joining the digits $x$ is on top of the strand joining the digits $y_i$.
The category generated by $\FH_n$ and $\BH$ is an indirect product $\FH_n\bowtie\BH$, since the action of $\FH_n$ preserves the subgroupoid $\BH$.
Indeed, the action of a generator of $\FH_n$ on $\BH$ is an expansion, i.e., it prompts a replacement of a strand joining the digits $y_i$ with two strands, one joining the digits $y_i$ and, immediately on its right, one joining new digits $x$;
thus, the action of a generator of $\FH_n$ on a generator of type 2 sends it to another generator of type 2 and the action on a generator of type 1 sends it to a product of a generator of type 1 and one of type 2.

\begin{figure}\centering
\begin{subfigure}[b]{.45\textwidth}
\centering
\begin{tikzpicture}
    \draw[line width=4pt,white] (2,-2) ellipse (2.7cm and .65cm); 
    \draw[gray] (2,-2) ellipse (2.7cm and .65cm); 
    \draw[line width=4pt,white] (.5,0) to[out=-90,in=90] (.5,-2);
    \draw[gray,dotted] (.5,0) to[out=-90,in=90] (.5,-2);
    \draw[line width=4pt,white] (1,0) to[out=-90,in=90] (1,-2);
    \draw[gray,dashed] (1,0) to[out=-90,in=90] (1,-2);
    \draw[line width=4pt,white] (2.3333,0) to [out=-90,in=90] (1.6667,-2);
    \draw[Plum] (2.3333,0) to [out=-90,in=90] (1.6667,-2);
    \draw[line width=4pt,white] (1.6667,0) to[out=-90,in=90] (2.3333,-2);
    \draw[black] (1.6667,0) to[out=-90,in=90] (2.3333,-2);
    \draw[line width=4pt,white] (3,0) to[out=-90,in=90] (3,-2);
    \draw[gray,dashed] (3,0) to[out=-90,in=90] (3,-2);
    \draw[line width=4pt,white] (3.5,0) to[out=-90,in=90] (3.5,-2);
    \draw[gray,dotted] (3.5,0) to[out=-90,in=90] (3.5,-2);
    \draw[line width=4pt,white] (2,0) ellipse (2.7cm and .65cm); 
    \draw[gray] (-.7,0) to (-.7,-2) (4.7,0) to (4.7,-2);
    \draw[gray] (2,0) ellipse (2.7cm and .65cm); 
    \draw[gray] (0,0) node[circle,inner sep=.8pt,fill]{};
    \draw[gray] (.5,0) node[circle,inner sep=.9pt,fill]{};
    \draw[gray] (1,0) node[circle,inner sep=1pt,fill]{};
    \draw[black] (1.6667,0) node[circle,inner sep=1.5pt,fill]{} node[above]{$x$};
    \draw[Plum] (2.3333,0) node[circle,inner sep=1.5pt,fill]{} node[above]{$y_i$};
    \draw[gray] (3,0) node[circle,inner sep=1pt,fill]{};
    \draw[gray] (3.5,0) node[circle,inner sep=.9pt,fill]{};
    \draw[gray] (4,0) node[circle,inner sep=.8pt,fill]{};
    \draw[gray] (0,-2) node[circle,inner sep=.8pt,fill]{};
    \draw[gray] (.5,-2) node[circle,inner sep=.9pt,fill]{};
    \draw[gray] (1,-2) node[circle,inner sep=1pt,fill]{};
    \draw[Plum] (1.6667,-2) node[circle,inner sep=1.5pt,fill]{} node[below]{$y_i$};
    \draw[black] (2.3333,-2) node[circle,inner sep=1.5pt,fill]{} node[below]{$x$};
    \draw[gray] (3,-2) node[circle,inner sep=1pt,fill]{};
    \draw[gray] (3.5,-2) node[circle,inner sep=.9pt,fill]{};
    \draw[gray] (4,-2) node[circle,inner sep=.8pt,fill]{};
\end{tikzpicture}
\caption{Generators of $\BH$ of type 1.}
\label{fig:braided:houghton:generators:type:1}
\end{subfigure}
\begin{subfigure}[b]{.525\textwidth}
\centering
\begin{tikzpicture}
    \draw[line width=4pt,white] (2.3333,-2) ellipse (3cm and .7cm); 
    \draw[gray] (2.3333,-2) ellipse (3cm and .7cm); 
    \draw[line width=4pt,white] (.5,0) to[out=-90,in=90] (.5,-2);
    \draw[gray,dotted] (.5,0) to[out=-90,in=90] (.5,-2);
    \draw[line width=4pt,white] (1.1667,0) to [out=-90,in=90] (1.1667,-2);
    \draw[Plum] (1.1667,0) to [out=-90,in=90] (1.1667,-2);
    \draw[line width=4pt,white,curvy] (1.8334,0) to (1.8334,-.75);
    \draw[line width=4pt,white,curvy] (1.8334,-1.25) to (1.8334,-2);
    \draw[black,curvy] (1.8334,0) to (1.8334,-.75);
    \draw[black,curvy] (1.8334,-1.25) to (1.8334,-2);
    \draw[line width=4pt,white,curvy] (2.3333,-.1) to (2.3333,-.75);
    \draw[line width=4pt,white,curvy] (2.3333,-1.25) to (2.3333,-1.9);
    \draw[black,curvy,dotted] (2.3333,-.1) to (2.3333,-.75);
    \draw[black,curvy,dotted] (2.3333,-1.25) to (2.3333,-1.9);
    \draw[line width=4pt,white,curvy] (2.8334,0) to (2.8334,-.75);
    \draw[line width=4pt,white,curvy] (2.8334,-1.25) to (2.8334,-2);
    \draw[black,curvy] (2.8334,0) to (2.8334,-.75);
    \draw[black,curvy] (2.8334,-1.25) to (2.8334,-2);
    \draw[line width=4pt,white] (3.5,0) to [out=-90,in=90] (3.5,-2);
    \draw[Orange] (3.5,0) to [out=-90,in=90] (3.5,-2);
    \draw[line width=4pt,white] (4.1667,0) to[out=-90,in=90] (4.1667,-2);
    \draw[gray,dotted] (4.1667,0) to[out=-90,in=90] (4.1667,-2);
    \draw[line width=4pt,white] (2.3333,0) ellipse (3cm and .7cm); 
    \draw[gray] (-.6667,0) to (-.6667,-2) (5.3333,0) to (5.3333,-2);
    \draw[gray] (2.3333,0) ellipse (3cm and .7cm); 
    \draw[gray] (0,0) node[circle,inner sep=.8pt,fill]{};
    \draw[gray] (.5,0) node[circle,inner sep=1pt,fill]{};
    \draw[Plum] (1.1667,0) node[circle,inner sep=1.5pt,fill]{} node[above]{$y_i$};
    \draw[black] (1.8334,0) node[circle,inner sep=1.25pt,fill]{} node[above]{$x$};
    \draw[black] (2.3333,0) node{$\dots$};
    \draw[black] (2.8333,0) node[circle,inner sep=1.25pt,fill]{} node[above]{$x$};
    \draw[Orange] (3.5,0) node[circle,inner sep=1.5pt,fill]{} node[above]{$y_{i+1}$};
    \draw[gray] (4.1667,0) node[circle,inner sep=1pt,fill]{};
    \draw[gray] (4.6667,0) node[circle,inner sep=.9pt,fill]{};
    \node at (2.3333,-1) {any $\sigma \in \mathrm{B}_k$};
    \draw[gray] (0,-2) node[circle,inner sep=.8pt,fill]{};
    \draw[gray] (.5,-2) node[circle,inner sep=1pt,fill]{};
    \draw[Plum] (1.1667,-2) node[circle,inner sep=1.5pt,fill]{} node[below]{$y_i$};
    \draw[black] (1.8334,-2) node[circle,inner sep=1.25pt,fill]{} node[below]{$x$};
    \draw[black] (2.3333,-2) node{$\dots$};
    \draw[black] (2.8333,-2) node[circle,inner sep=1.25pt,fill]{} node[below]{$x$};
    \draw[Orange] (3.5,-2) node[circle,inner sep=1.5pt,fill]{} node[below]{$y_{i+1}$};
    \draw[gray] (4.1667,-2) node[circle,inner sep=1pt,fill]{};
    \draw[gray] (4.6667,-2) node[circle,inner sep=.9pt,fill]{};
\end{tikzpicture}
\caption{Generators of $\BH$ of type 2.}
\label{fig:braided:houghton:generators:type:2}
\end{subfigure}
\caption{The generators that define $\BH$.}
\label{fig:braided:houghton:generators}
\end{figure}

Note that, for a morphism in $\BH(w,w)$, if we consider a curve enclosing all the $x$'s and such that its cylinder encloses all the corresponding strands, then the braid defined in such cylinder belongs to the braid group $\mathrm{B}_k$ with $k$ being the number of $x$'s in $w$.
Everything else about such morphism (i.e., all about the strands joining the digits $y_i$) is then automatically determined, since each strand joining the digits $y_i$ must pass behind the cylinder and cannot cross themselves.
Moreover, expansions of strands joining the digits $y_i$ cause the appearance of unbraided strands joining the digits $x$ on the immediate right of $y_i$, before the next $y_{i+1}$ (if there is one).
Thus, in any representative of an element of $\Pi(\FH_n\bowtie\BH)$, each strand joining the digits $y_i$ represents a ``tight combing'' (``straff gek\"ammt'' at pages 2 and 3 in \cite{braidedHoughtonThesis}), which is essentially an unbraided translation of the digits $x$ that are obtained by the expansions of a fixed $y_i$. 
The fundamental group of $\FH_n\bowtie\BH$ is thus the group of braids between the digits $x$ that are unbraided translations out of a finite subset, and is thus isomorphic to the braided Houghton group.

\begin{figure}\centering
\begin{tikzpicture}
    \draw[line width=3pt,white] (2.5,-2) ellipse (1cm and .4cm); 
    \draw[line width=4pt,white] (2,-2) ellipse (2.7cm and .65cm); 
    \draw[gray] (2,-2) ellipse (2.7cm and .65cm); 
    \draw[gray,dotted] (2.5,-2) ellipse (1cm and .4cm); 
    \draw[line width=4pt,white] (0,0) to[out=-90,in=90] (0,-2);
    \draw[red] (0,0) to[out=-90,in=90] (0,-2);
    \draw[line width=4pt,white] (2,0) to[out=-90,in=90] (1,-2);
    \draw[blue] (2,0) to[out=-90,in=90] (1,-2);
    \draw[line width=4pt,white] (1,0) to [out=-90,in=90] (3,-2);
    \draw[black] (1,0) to[out=-90,in=90] (3,-2);
    \draw[line width=4pt,white] (3,0) to [out=-90,in=90] (4,-2);
    \draw[ForestGreen] (3,0) to[out=-90,in=90] (4,-2);
    \draw[line width=4pt,white] (4,0) to [out=-90,in=90] (2,-2);
    \draw[black] (4,0) to[out=-90,in=90] (2,-2);
    \draw[line width=4pt,white] (2,0) ellipse (2.7cm and .65cm); 
    \draw[line width=3pt,white] (.5,0) arc(180:0:.5cm and .2cm) (1.5,0) arc(180:360:1cm and .4cm) (3.5,0) arc(180:0:.5cm and .24cm) (4.5,0) arc(0:-180:2cm and .5cm); 
    \draw[gray] (-.7,0) to (-.7,-2) (4.7,0) to (4.7,-2);
    \draw[gray] (2,0) ellipse (2.7cm and .65cm); 
    \draw[gray,dotted] (.5,0) arc(180:0:.5cm and .2cm) (1.5,0) arc(180:360:1cm and .4cm) (3.5,0) arc(180:0:.5cm and .24cm) (4.5,0) arc(0:-180:2cm and .5cm); 
    \draw[red] (0,0) node[circle,inner sep=1.5pt,fill]{} node[left]{$y_1$};
    \draw[black] (1,0) node[circle,inner sep=1.5pt,fill]{} node[left]{$x$};
    \draw[blue] (2,0) node[circle,inner sep=1.5pt,fill]{} node[left]{$y_2$};
    \draw[ForestGreen] (3,0) node[circle,inner sep=1.5pt,fill]{} node[left]{$y_3$};
    \draw[black] (4,0) node[circle,inner sep=1.5pt,fill]{} node[left]{$x$};
    \draw[red] (0,-2) node[circle,inner sep=1.5pt,fill]{} node[left]{$y_1$};
    \draw[blue] (1,-2) node[circle,inner sep=1.5pt,fill]{} node[left]{$y_2$};
    \draw[black] (2,-2) node[circle,inner sep=1.5pt,fill]{} node[left]{$x$};
    \draw[black] (3,-2) node[circle,inner sep=1.5pt,fill]{} node[left]{$x$};
    \draw[ForestGreen] (4,-2) node[circle,inner sep=1.5pt,fill]{} node[left]{$y_3$};
    \begin{scope}[xshift=6.5cm]
    \draw[line width=3pt,white] (2.5,-2) ellipse (1cm and .4cm); 
    \draw[gray,dotted] (2.5,-2) ellipse (1cm and .4cm); 
    \draw[line width=4pt,white] (2,-2) ellipse (2.7cm and .65cm); 
    \draw[gray] (2,-2) ellipse (2.7cm and .65cm); 
    \draw[line width=4pt,white] (0,0) to[out=-90,in=90] (0,-2);
    \draw[red] (0,0) to[out=-90,in=90] (0,-2);
    \draw[line width=4pt,white] (2,0) to[out=-90,in=90] (1,-2);
    \draw[blue] (2,0) to[out=-90,in=90] (1,-2);
    \draw[line width=4pt,white] (1,0) to [out=-90,in=90] (3,-2);
    \draw[black] (1,0) to[out=-90,in=90] (3,-2);
    \draw[line width=4pt,white] (4,0) to [out=-90,in=90] (2,-2);
    \draw[black] (4,0) to[out=-90,in=90] (2,-2);
    \draw[line width=4pt,white] (3,0) to [out=-90,in=90] (4,-2);
    \draw[ForestGreen] (3,0) to[out=-90,in=90] (4,-2);
    \draw[line width=4pt,white] (2,0) ellipse (2.7cm and .65cm); 
    \draw[line width=3pt,white] (0.5,0) arc(180:0:.5cm and .2cm) (1.5,0) arc(180:270:.5cm and .25cm) (2,-0.25) arc(270:360:.35cm and .25cm) (2.35,0) arc(180:0:1.075cm and .4cm) (3.5,0) arc(-180:0:.5cm and .2cm) (2.5,0) arc(0:-180:1cm and .5cm) (2.5,0) arc(180:0:.5cm and .2cm); 
    \draw[gray] (-.7,0) to (-.7,-2) (4.7,0) to (4.7,-2);
    \draw[gray] (2,0) ellipse (2.7cm and .65cm); 
    \draw[gray,dotted] (0.5,0) arc(180:0:.5cm and .2cm) (1.5,0) arc(180:270:.5cm and .25cm) (2,-0.25) arc(270:360:.35cm and .25cm) (2.35,0) arc(180:0:1.075cm and .4cm) (3.5,0) arc(-180:0:.5cm and .2cm) (2.5,0) arc(0:-180:1cm and .5cm) (2.5,0) arc(180:0:.5cm and .2cm); 
    \draw[red] (0,0) node[circle,inner sep=1.5pt,fill]{} node[left]{$y_1$};
    \draw[black] (1,0) node[circle,inner sep=1.5pt,fill]{} node[left]{$x$};
    \draw[blue] (2,0) node[circle,inner sep=1.5pt,fill]{} node[left]{$y_2$};
    \draw[ForestGreen] (3,0) node[circle,inner sep=1.5pt,fill]{} node[left]{$y_3$};
    \draw[black] (4,0) node[circle,inner sep=1.5pt,fill]{} node[left]{$x$};
    \draw[red] (0,-2) node[circle,inner sep=1.5pt,fill]{} node[left]{$y_1$};
    \draw[blue] (1,-2) node[circle,inner sep=1.5pt,fill]{} node[left]{$y_2$};
    \draw[black] (2,-2) node[circle,inner sep=1.5pt,fill]{} node[left]{$x$};
    \draw[black] (3,-2) node[circle,inner sep=1.5pt,fill]{} node[left]{$x$};
    \draw[ForestGreen] (4,-2) node[circle,inner sep=1.5pt,fill]{} node[left]{$y_3$};
    \end{scope}
\end{tikzpicture}
\caption{An element of $\BH(\textcolor{red}{y_1} x \textcolor{blue}{y_2} \textcolor{ForestGreen}{y_3} x, \textcolor{red}{y_1} \textcolor{blue}{y_2} x x \textcolor{ForestGreen}{y_3})$ next to one of $\B(\textcolor{red}{y_1} x \textcolor{blue}{y_2} \textcolor{ForestGreen}{y_3} x, \textcolor{red}{y_1} \textcolor{blue}{y_2} x x \textcolor{ForestGreen}{y_3}) \setminus \BH(\textcolor{red}{y_1} x \textcolor{blue}{y_2} \textcolor{ForestGreen}{y_3} x, \textcolor{red}{y_1} \textcolor{blue}{y_2} x x \textcolor{ForestGreen}{y_3})$.}
\label{fig:braided:houghton}
\end{figure}


\section*{Acknowledgements}

The authors would like to thank Mar\'{i}a Cumplido for bringing the topic of orderability in Thompson-like groups to their attention and are grateful to Mar\'{i}a Cumplido, Juan Gonz\'alez-Meneses and Anne Lonjou for useful conversations and advices and to Kai-Uwe Bux for providing the authors with Degenhardt's PhD thesis \cite{braidedHoughtonThesis}.
The authors are thankful to the anonymous referee, whose comments prompted improvements in the exposition and removals of many typos.


\emergencystretch=2em
\printbibliography[heading=bibintoc]

\end{document}